\documentclass{amsart}
\usepackage[margin=1in]{geometry}

\usepackage{color,amssymb, comment, mathrsfs,tikz,upgreek,contour, soul, colonequals, mathdots,fancyvrb,amsmath,bm}
\usepackage[colorlinks, linkcolor={blue}]{hyperref}
\usepackage{tikz-cd}
\usepackage{cleveref}
\usetikzlibrary{cd}
\usetikzlibrary{fit,shapes.geometric}
\usetikzlibrary{matrix}
\usetikzlibrary{arrows}
\usepackage{enumitem}
\usepackage{mathtools}
\usepackage[all]{xy}
\usepackage[mathlines,pagewise]{lineno}
\usepackage{cancel}
\allowdisplaybreaks

\newtheorem{theorem}{Theorem}[section]
\newtheorem{lemma}[theorem]{Lemma}
\newtheorem{proposition}[theorem]{Proposition}
\newtheorem{corollary}[theorem]{Corollary}

\usepackage[only,llbracket,rrbracket,llparenthesis,rrparenthesis]{stmaryrd} 
\usepackage{accsupp}

\theoremstyle{definition}
\newtheorem{definition}[theorem]{Definition}
\newtheorem{construction}[theorem]{Construction}
\newtheorem{example}[theorem]{Example}

\theoremstyle{remark}
\newtheorem{remark}[theorem]{Remark}

\newtheorem{chunk}[theorem]{}

\numberwithin{equation}{section}

\newcommand{\kk}{\Bbbk}

\newcommand{\supp}{\operatorname{supp}}

\makeatletter
\DeclareFontEncoding{LS2}{}{\@noaccents}
\makeatother
\DeclareFontSubstitution{LS2}{stix}{m}{n}

\DeclareSymbolFont{largesymbolsstix}{LS2}{stixex}{m}{n}

\DeclareMathDelimiter{\lbrbrak}{\mathopen}{largesymbolsstix}{"EE}{largesymbolsstix}{"14}
\DeclareMathDelimiter{\rbrbrak}{\mathclose}{largesymbolsstix}{"EF}{largesymbolsstix}{"15}

\crefname{diagram}{diagram}{diagrams}
\crefname{diagram}{Diagram}{Diagrams}
\creflabelformat{diagram}{(#1) #2 #3}

\renewcommand{\b}{\bullet}
\newcommand{\cone}{\operatorname{cone}}

\newcommand{\set}{\operatorname{set}}
\newcommand{\bbT}{\mathbb{T}}
\newcommand{\bbK}{\mathbb{K}}
\newcommand{\tbbT}{\widetilde{\mathbb{T}}}
\renewcommand{\b}{\bullet}
\newcommand{\bbHT}{\mathbb{HT}}

\begin{document}
\title[The Herzog-Takayama resolution over a skew polynomial ring]{The Herzog-Takayama resolution over a skew polynomial ring}
\author[L.~Ferraro]{Luigi Ferraro}

\author[L. Utkina]{Linoy Utkina}

\begin{abstract}
Let $\kk$ be a field, and let $I$ be a monomial ideal in the polynomial ring $R=\kk[x_1,\ldots,x_n]$. In her thesis, Taylor introduced a complex that provides a finite free resolution of $R/I$ as an $R$-module. Building on this, Ferraro, Martin and Moore extended this construction to monomial ideals in skew polynomial rings. Since the Taylor resolution is generally not minimal, significant effort has been devoted to identifying classes of ideals with minimal free resolutions that are relatively straightforward to construct. In a 1987 paper, Eliahou and Kervaire developed a minimal free resolution for a class of monomial ideals in $R$ known as stable ideals. This result was later generalized to stable ideals in skew polynomial rings by Ferraro and Hardesty. In a 2002 paper, Herzog and Takayama constructed a minimal free resolution for monomial ideals with linear quotients, a broader class of ideals containing stable ideals. Their resolution reduces to the Eliahou-Kervaire resolution in the stable case. In this paper, we generalize the Herzog-Takayama resolution to skew polynomial rings.

\end{abstract}

\maketitle

\section{Introduction}
Let $\kk$ be a field, and let $R$ denote a standard graded commutative polynomial ring over $\kk$. In commutative algebra, a fundamental topic of study is the free resolutions of monomial ideals in $R$. Taylor's construction in \cite{TaylorDiana} provides a free resolution for any monomial ideal, though this resolution is typically not minimal. Various minimal free resolutions have been developed for specific types of monomial ideals. Examples include the Eliahou-Kervaire resolution for stable ideals \cite{EK} and the Aramova-Herzog-Hibi resolution for squarefree stable ideals \cite{sqrfreeStable}. Reiner and Welker \cite{ReinWelk}, along with Novik, Postnikov, and Sturmfels \cite{NovPosSturm}, have further constructed minimal free resolutions for matroidal ideals. These classes of monomial ideals fall within the framework introduced by Herzog and Takayama in \cite{cone}; specifically, they are ideals with linear quotients that admit a regular decomposition function, for which Herzog and Takayama provide a minimal free resolution.

Skew polynomial rings, defined as polynomial rings in which variables commute up to multiplication by a nonzero scalar, are of interest in noncommutative algebra and noncommutative algebraic geometry. The construction of free resolutions, a central problem in commutative settings, naturally extends to these noncommutative contexts. The first author, together with Martin and Moore, generalized Taylor's construction to obtain a free resolution for monomial ideals in skew polynomial rings \cite{Taylor}. In a subsequent work, the first author and Hardesty, extended the Eliahou-Kervaire resolution to stable ideals within skew polynomial rings \cite{EKskew}. The objective of this paper is to further generalize the Herzog-Takayama resolution, adapting it to ideals with linear quotients and regular decomposition function in skew polynomial rings. This construction generalizes the one developed by the first author and Hardesty in \cite{EKskew} and provides a minimal free resolution for both squarefree stable ideals and matroidal ideals in the noncommutative case.

The construction of the skew Eliahou-Kervaire resolution in \cite{EKskew} is based on the original work by Eliahou and Kervaire \cite{EK}. Similarly, the skew Herzog-Takayama resolution we present here draws on ideas from \cite{cone}, namely constructing the resolution as an iterated mapping cone. This approach applies to all ideals with linear quotients; however, to explicitly determine a formula for the differential, we further assume that the decomposition function of the ideal is regular, a concept introduced by Herzog and Takayama in \cite{cone}.

The paper is organized as follows. In Section 2, we recall the definitions and properties of two group bicharacters introduced in \cite{Taylor} and \cite{FM}, which will play a key role in our constructions. We then adapt the definition of ideals with linear quotients and regular decomposition function from \cite{cone} to the skew case, and we state preliminary lemmas that mirror their commutative counterparts from \cite{cone}, with proofs omitted due to similarity.

Section 3 begins by defining the skew Taylor resolution twisted by a monomial, which servers as a tool for constructing the skew Herzog-Takayama resolution. We show that this twisted skew Taylor resolution is isomorphic to the untwisted skew Taylor resolution, but we choose the twisted form as it yields a differential in the skew Herzog-Takayama resolution that resembles that of the skew Eliahou-Kervaire resolution from \cite{EKskew}. Additionally, we prove that when the sequence of monomials generating the ideal is regular, the skew Taylor resolution is isomorphic to the skew Koszul complex defined in \cite{FM}. This parallels the approach in \cite{cone}, where the Koszul complex on the variables (which in the commutative case coincides with the Taylor resolution) is employed. We prefer the twisted skew Taylor resolution over the skew Koszul complex here, as it produces a more refined differential for the skew Herzog-Takayama resolution.

Finally, we define the skew Herzog-Takayama resolution and prove that it minimally resolves ideals with linear quotients that admit a regular decomposition function. We conclude the paper with an example and several corollaries concerning Poincaré series, Betti numbers, projective dimension, and regularity.

\section{Background and Preliminaries}

In this paper, we work over a field $\kk$ and consider the skew polynomial ring $R=\kk_\mathfrak{q}[x_1,\ldots,x_n]$, where the multiplication is defined by the relation $x_ix_j=q_{i,j}x_jx_i$ for all $i,j$. Here each $q_{j,i}$ belongs to $\kk^*$ and satisfies $q_{j,i}=q_{i,j}^{-1}$. Additionally, the field $\kk$ is central in $R$.

We will denote by $X$ the set of monomials in $R$. If $\mathbf{a}=(a_1,\ldots,a_n)\in\mathbb{N}^n$, we use the notation $\mathbf{x}^\mathbf{a}$ to represent the monomial $x_1^{a_1}\cdots x_n^{a_n}$. Given $\mathbf{a}=(a_1,\ldots,a_n)$ and $\mathbf{b}=(b_1,\ldots,b_n)\in\mathbb{N}^n$ with $a_i\geq b_i$ for all $i=1,\ldots,n$, we denote by $\frac{\mathbf{x}^\mathbf{a}}{\mathbf{x}^{\mathbf{b}}}$ the monomial $\mathbf{x}^{\mathbf{a}-\mathbf{b}}$. Likewise, the product $\mathbf{x}^{\mathbf{a}}*\mathbf{x}^\mathbf{b}$ represents the monomial $\mathbf{x}^{\mathbf{a}+\mathbf{b}}$. The set $X$, equipped with the operation $*$, forms a monoid.

The ring $R$ has a $\mathbb{Z}$-grading defined by assigning each variable $x_i$ a degree $\mathrm{deg}\;x_i=d_i\in\mathbb{Z}^+$ for all $i=1,\ldots,n$. Let $\sigma_1,\ldots,\sigma_n$ denote the normalizing automorphisms of the variables $x_1,\ldots, x_n$ in $R$. These automorphisms generate an abelian subgroup $G$ of the group of $\mathbb{Z}$-graded ring automorphisms of $R$. Consequently, $R$ admits a $G\times\mathbb{Z}$-grading, where the $G$-degree of a monomial $x_1^{a_1}\cdots x_n^{a_n}$ is given by $\sigma_1^{a_1}\cdots\sigma_n^{a_n}$.

The mapping $(\sigma_i,\sigma_j)\mapsto q_{i,j}$ extends to an alternating bicharacter $\chi:G\times G\rightarrow\kk^*$, as noted in \cite[Remark 2.1]{Taylor}. If two monomials $\mathbf{x}^\mathbf{a}$ and $\mathbf{x}^\mathbf{b}$ have $G$-degree $\sigma \text{ and }\tau$, we use the notation $\chi(\mathbf{x}^\mathbf{a},\mathbf{x}^\mathbf{b})$ to denote $\chi(\sigma,\tau)$. This implies that $\mathbf{x}^\mathbf{a}\mathbf{x}^\mathbf{b}=\chi(\mathbf{x}^\mathbf{a},\mathbf{x}^\mathbf{b})\mathbf{x}^\mathbf{b}\mathbf{x}^\mathbf{a}$. 
We define a function $C:X\times X\rightarrow \kk^*$ by equating
\[
\mathbf{x}^\mathbf{a}\mathbf{x}^\mathbf{b}=C(\mathbf{x}^\mathbf{a},\mathbf{x}^\mathbf{b})\mathbf{x}^\mathbf{a}*\mathbf{x}^\mathbf{b}.
\] 
Note that $C$ forms a bicharacter of the monoid $(X,*)$. Furthermore, \cite[Lemma 2.6(c)]{Taylor} established that
\begin{equation}\label{eq:CChiRel}
\frac{C(\mathbf{x}^\mathbf{a},\mathbf{x}^\mathbf{b})}{C(\mathbf{x}^\mathbf{b},\mathbf{x}^\mathbf{a})}=\chi(\mathbf{x}^\mathbf{a},\mathbf{x}^\mathbf{b}).
\end{equation}

\hspace{1cm}

\begin{remark}
We extend $C$ to the group of monomials in $x_1^{\pm1},\ldots, x_n^{\pm1}$ by
\[
C(\mathbf{x}^{\mathbf{\bm{a}}-\mathbf{\bm{b}}},\mathbf{x}^{\mathbf{\bm{c}}-\mathbf{\bm{d}}})=C(\mathbf{x}^\mathbf{\bm{a}},\mathbf{x}^\mathbf{\bm{c}})C(\mathbf{x}^\mathbf{\bm{a}},\mathbf{x}^\mathbf{\bm{d}})^{-1}C(\mathbf{x}^\mathbf{\bm{b}},\mathbf{x}^\mathbf{\bm{c}})^{-1}C(\mathbf{x}^\mathbf{\bm{b}},\mathbf{x}^\mathbf{\bm{d}}),
\]
where $\mathbf{\bm{a,b,c,d}}\in\mathbb{N}^n$.
\end{remark}

\begin{remark}\label{rem:FrankBook}
It is proved in \cite[Theorem 1.3.6]{FrankBook} that every monomial ideal in a commutative polynomial ring with coefficients in a commutative ring admits a unique set of minimal generators. This proof can be adapted to our context and it is therefore omitted. Moreover in \cite[Theorem 1.1.9]{FrankBook} it is proved that if a monomial belongs to a monomial ideal, then it is a multiple of a minimal generator of the monomial ideal. This proof can also be adapted to our context and it is therefore omitted. 
\end{remark}
Throughout the paper, the $R$-modules are assumed to be $G\times\mathbb{Z}$-graded \emph{right} $R$-modules, which can be turned into bimodules with the following left action
\[
r\cdot m\colonequals \chi(r,m)m\cdot r,
\]
where $r$ is a homogeneous element of $R$ and $m$ is a homogeneous element of an $R$-module. By abuse of notation $\chi(r,m)$ denotes the bicharacter $\chi$ applied to the $G$-degree of $r$ and $m$. We also point out that the elements of $\kk$ can freely move from left to right (and vice versa) when multiplied by an element of a module.

\begin{remark}
All resolutions and complexes over $R$ constructed in this paper have a natural $G\times\mathbb{Z}$-grading for which the differentials become homogeneous with respect to the grading. This grading refines the $\mathbb{Z}$-grading that one is accustomed to in the commutative case.
\end{remark}

\begin{definition}
A monomial ideal $I$ of $R$ is said to have linear quotients if there exists an order on its minimal generators $u_1,\ldots,u_m$ such that the colon ideals 
\[
((u_1,\ldots,u_j):u_{j+1})
\]
are generated by variables for all $j=1,\ldots,m-1$. In this case, if the colon ideal above is generated by the variables $x_{i_1},\ldots,x_{i_l}$, then we define
\[
\set(u_{j+1})\colonequals\{i_1,\ldots,i_l\}.
\]
\end{definition}
We also recall that if $u$ is a monomial, then $\supp(u)=\{i\mid x_i\text{ divides }u\}$.

\begin{example}
The stable ideals defined in \cite[Definition 2.2]{EKskew} are examples of ideals with linear quotients.
\end{example}

Let $I=(u_1,\ldots, u_m)$ be an ideal of $R$ with linear quotients with respect to the given order of the minimal generators. Set $I_j=(u_1,\ldots,u_j)$ for $j=1,\ldots,m$. Let $M(I)$ be the set of all monomials in $I$ and $G(I)$ the set of minimal generators of $I$. We define a map $g:M(I)\rightarrow G(I)$ as $g(u)=u_j$, where $j$ is the smallest integer such that $u\in I_j$. The function $g$ is called the \emph{decomposition function} for $I$. The \emph{complement function} of $I$ is the function $\kappa$ such that $u=g(u)*\kappa(u)$. 

The next lemmas can be proved as in the commutative case and their proofs are therefore omitted, see \cite[Lemma 1.7, Lemma 1.8 and Lemma 1.11]{cone}.

\begin{lemma}
\begin{enumerate}
\item For all $u\in M(I)$ one has $u=g(u)*\kappa(u)$ with
\[
\set(g(u))\cap\supp(\kappa(u))=\emptyset.
\]
\item Let $u$ be a monomial in $I$ such that $u=v*w$ with $v\in G(I)$ and $w$ a monomial such that $\set(v)\cap\supp(w)=\emptyset$. Then $v=g(u)$.
\end{enumerate}
\end{lemma}

\begin{lemma}\label{lem:lem1.8}
Let $u,v\in M(I)$. Then $g(u*v)=g(u)$ if and only if
\[
\set(g(u))\cap\supp(v)=\emptyset.
\]
\end{lemma}

\begin{definition}
We say that the decomposition function $g:M(I)\rightarrow G(I)$ is \emph{regular}, if $\set(g(x_s*u))\subseteq\set(u)$ for all $s\in\set(u)$ and $u\in G(I)$.
\end{definition}

\begin{lemma}\label{lem:swap}
If $g:M(I)\rightarrow G(I)$ is a regular decomposition function, then
\[
g(x_s*g(x_t*u))=g(x_s*x_t*u)=g(x_t*g(x_s*u)),\quad\text{for all}\;u\in M(I)\;\text{and all}\;s,t\in\set(u).
\]
\end{lemma}

\section{The Herzog-Takayama resolution}
With the notation set at the beginning of the previous section, let $L_j$ be the ideal $(I_j:u_{j+1})$ for $j=1,\ldots, m-1$. For given monomials $m_1,\ldots,m_s$, let $\bbT_\b(m_1,\ldots,m_s)$ be the skew Taylor resolution of the ideal $(m_1,\ldots,m_s)$ as defined in \cite[Construction 3.1]{Taylor}. We will be denoting by $e_\sigma$ with $\sigma\subseteq[s]$ the basis elements of $\bbT_\b$.

\begin{remark}
We recall that the least common multiple of the monomials $\mathbf{x}^\mathbf{a},\mathbf{x}^\mathbf{b}$ in the skew polynomial ring $R$ is $\mathbf{x}^\mathbf{c}$ where $c_i=\max\{a_i,b_i\}$. Similarly for the least common multiple of more than two monomials.
\end{remark}

\begin{definition}
Let $m_1,\ldots,m_s$ be monomials minimally generating an ideal $I$ and let $u$ be any monomial. The \emph{skew Taylor resolution $\tbbT_\b(m_1,\ldots,m_s;u)$ of $I$ twisted by $u$} is the resolution defined as $\tbbT_i=R^{\binom{s}{i}}$ and, denoting by $e(\sigma;u)$ with $\sigma\subseteq[s]$ a basis of $\tbbT_i$, with differential given by
\[
\partial^{\tbbT}(e(\sigma;u))= \sum_{i \in \sigma} e(\sigma_i;u) (-1)^{\alpha(\sigma;i)} C\left(m_{\sigma_i}*u, \frac{m_\sigma}{m_{\sigma_i}}\right)^{-1} \frac{m_\sigma}{m_{\sigma_i}},
\]
where $\alpha(\sigma;i)=|\{j\in\sigma\mid j<i\}|$, the monomial $m_\sigma$ is the least common multiple of the monomials $m_i$ with $i\in\sigma$, and $\sigma_i$ denotes the set $\sigma$ with $i$ removed.
\end{definition}
We point out that the skew Taylor resolution $\bbT_\b(m_1,\ldots,m_s)$ is just $\tbbT_\b(m_1,\ldots,m_s;1)$, in which case the basis elements $e(\sigma;1)$ will just be denoted by $e_\sigma$. To prove that $\tbbT_\b(m_1,\ldots,m_s;u)$ is indeed a resolution (and a complex), we prove that it is isomorphic to $\bbT_\b(m_1,\ldots,m_s)$.

\begin{proposition}\label{prop:twisted}
The complexes $\tbbT_\b(m_1,\ldots,m_s;u)$ and $\bbT_\b(m_1,\ldots,m_s)$ are isomorphic.
\end{proposition}
\begin{proof}
Let $\varphi:\bbT_\b\rightarrow\tbbT_\b$ be the map defined as $\varphi(e_\sigma)=e(\sigma;u)C(u,m_\sigma)$. This map is clearly invertible, we check that this map commutes with the differential.\\
Applying $\varphi$ to $\partial^{\bbT}(e_{\sigma})$ yields
\begin{equation*}
    \begin{split}
        \varphi (\partial^{\bbT}(e_{\sigma})) & = \varphi\left(\sum_{i \in \sigma} e_{\sigma_i}(-1)^{\alpha(\sigma;i)} C \left(m_{\sigma_i},\frac{m_{\sigma}}{m_{\sigma_i}}\right)^{-1} \frac{m_{\sigma}}{m_{\sigma_i}} \right) \\
        & = \sum_{i\in\sigma} e(\sigma_i;u)(-1)^{\alpha(\sigma;i)} C(u, m_{\sigma_i}) C \left(m_{\sigma_i},\frac{m_{\sigma}}{m_{\sigma_i}}\right)^{-1} \frac{m_{\sigma}}{m_{\sigma_i}}.
    \end{split}
\end{equation*}
Applying $\partial^{\tbbT}$ to $\varphi(e_{\sigma})$ yields
\begin{equation*}
    \begin{split}
        \partial^{\tbbT}(\varphi(e_{\sigma})) & = \partial^{\tbbT} \left( C(u,m_{\sigma})e(\sigma;u) \right)\\
        & = \sum_{i\in\sigma} e(\sigma_i;u)(-1)^{\alpha(\sigma;i)} C(u, m_{\sigma}) C \left(m_{\sigma_i}*u,\frac{m_{\sigma}}{m_{\sigma_i}}\right)^{-1} \frac{m_{\sigma}}{m_{\sigma_i}} .
    \end{split}
\end{equation*}
To finish the proof we need to show that

\[C(u, m_{\sigma_i}) C\left(m_{\sigma_i},\frac{m_{\sigma}}{m_{\sigma_i}}\right)^{-1} = C(u, m_{\sigma}) C\left(m_{\sigma_i}*u,\frac{m_{\sigma}}{m_{\sigma_i}}\right)^{-1}.
\]
Indeed, the previous equality is equivalent to the following ones
\begin{align*}
C(u, m_{\sigma})C(u, m_{\sigma_i})^{-1} &= C\left(m_{\sigma_i} * u,\frac{m_{\sigma}}{m_{\sigma_i}}\right)C\left(m_{\sigma_i},\frac{m_{\sigma}}{m_{\sigma_i}}\right)^{-1}\\
C(u, m_{\sigma})C(u, m_{\sigma_i}^{-1}) &= C\left(m_{\sigma_i} * u,\frac{m_{\sigma}}{m_{\sigma_i}}\right)C\left(m_{\sigma_i}^{-1},\frac{m_{\sigma}}{m_{\sigma_i}}\right)\\
C\left(u,\frac{m_{\sigma}}{m_{\sigma_i}}\right) &= C\left(\frac{m_{\sigma_i}*u}{m_{\sigma_i}},\frac{m_{\sigma}}{m_{\sigma_i}}\right)\\
C\left(u,\frac{m_{\sigma}}{m_{\sigma_i}}\right) &= C\left(u,\frac{m_{\sigma}}{m_{\sigma_i}}\right)\qedhere
\end{align*}
\end{proof}

\begin{remark}
Let $s$ be a positive integer. It follows from \cite[Construction 2.10]{FM} that the skew Koszul complex on the monomials $m_1,\ldots,m_s$ is the complex $\bbK_\b(m_1,\ldots,m_s)$ defined as follows: in degree $i$ one has $\bbK_i=R^{\binom{s}{i}}$ with a basis given by $e_\sigma$ where $\sigma$ is a subset of $[s]$ of size $i$. The differential is given by
\[
\partial^\bbK(e_\sigma)=\sum_{r=1}^j(-1)^{r-1}e_{\sigma_{i_r}}\chi(m_{i_r},m_{i_{r+1}}*\cdots*m_{i_j}) m_{i_r},
\]
where $\sigma=\{i_1,\ldots,i_j\}$ is a subset of $[s]$ with $i_1<\cdots< i_j$. When $r=j$ in the previous summation, the second argument of $\chi$ is $1$.
\end{remark}

In the commutative case, the Koszul complex on a set of monomials with disjoint support and the corresponding Taylor resolution coincide. In the skew case these complexes do not coincide in general, but they are isomorphic, as the next Proposition shows. In the proof of \Cref{thm:main}, we will resolve modules of the form $R/I$, where $I$ is generated by variables, by using a twisted skew Taylor resolution, instead of a skew Koszul complex. If we were to use the skew Koszul complex, we would get a resolution isomorphic to the skew Herzog-Takayama resolution (defined below), but with a less pleasing differential.

\begin{proposition}
If the monomials $m_1,\ldots,m_s$ have disjoint support, then the complexes $\bbK_\b(m_1,\ldots,m_s)$ and $\bbT_\b(m_1,\ldots,m_s)$ are isomorphic.
\end{proposition}

\begin{proof}
We define a map $\psi:\bbT_\b\rightarrow\bbK_\b$ as follows: if $\sigma=\{i_1,\ldots,i_j\}$ with $i_1<\cdots<i_j$, then
\[
\psi(e_\sigma)=e_\sigma\prod_{r=2}^jC(m_{i_1}*\cdots *m_{i_{r-1}},m_{i_r})^{-1}.
\]
It is clear that $\psi$ is bijective, we prove that it is a chain map. Below we make use of the fact that the least common multiple of monomials with disjoint supports is the product of the monomials. Indeed, 

\begin{align*}
    \psi \partial^{\bbT}(e_\sigma) &= \psi\left(\sum_{r=1}^{j}e_{\sigma_{i_r}} (-1)^{r-1}C(m_{i_1}*\cdots *m_{i_{r-1}}*m_{i_{r+1}}*\cdots *m_{i_j}, m_{i_r})^{-1}m_{i_r}\right)\\
    &= \sum_{r=1}^{j}e_{\sigma_{i_r}} (-1)^{r-1}\prod_{l=2}^{r-1}C(m_{i_1}*\cdots*m_{i_{l-1}}, m_{i_l})^{-1}\\
    &\phantom{====}\cdot\prod_{l=r+1}^{j}C(m_{i_1}*\cdots * m_{i_{r-1}}*m_{i_{r+1}}*\cdots *m_{i_{l-1}}, m_{i_l})^{-1}\\ 
  &\phantom{====} \cdot C(m_{i_1}*\cdots *m_{i_{r-1}}*m_{i_{r+1}}*\cdots *m_{i_j}, m_{i_r})^{-1}m_{i_r}\\
      &= \sum_{r=1}^{j}e_{\sigma_{i_r}} (-1)^{r-1}\prod_{l=2}^{r-1}C(m_{i_1}*\cdots *m_{i_{l-1}}, m_{i_l})^{-1} \\
&\phantom{====}\cdot\prod_{l=r+1}^{j}C(m_{i_1}*\cdots *m_{i_{r-1}}*m_{i_{r+1}}*\cdots*m_{i_{l-1}}, m_{i_l})^{-1}\\ 
  &\phantom{====} \cdot C(m_{i_1}*\cdots *m_{i_{r-1}}, m_{i_r})^{-1}C(m_{i_{r+1}}*\cdots *m_{i_j}, m_{i_r})^{-1}m_{i_r}.\\ 
\end{align*}

Now we compute $\partial^\bbK(\psi(e_\sigma))$ and show that it coincides with $\psi(\partial^\bbT(e_\sigma))$. Indeed, using \eqref{eq:CChiRel} to write $\partial^\bbK$ in terms of the bicharacter $C$, we get that $\partial^{\bbK}(\psi(e_\sigma))$ is equal to

\begin{align*}
    &\phantom{=} \partial^{\bbK}\left(e_\sigma\prod_{l=2}^{j} C(m_{i_1}*\cdots* m_{i_{l-1}}, m_{i_l})^{-1}  \right)\\
    &= \sum_{r=1}^{j}(-1)^{r-1}e_{\sigma_{i_r}}\left(\prod_{l=2}^{j} C(m_{i_1}*\cdots *m_{i_{l-1}}, m_{i_l})^{-1}\right)\\
    &\phantom{==}\cdot C(m_{i_r}, m_{i_{r+1}}* \cdots * m_{i_j})C(m_{i_{r+1}}*\cdots *m_{i_j}, m_{i_r})^{-1}m_{i_r}\\
    &= \sum_{r=1}^{j}(-1)^{r-1}e_{\sigma_{i_r}}\\
    &\phantom{==}\cdot\left(C(m_{i_1}*\cdots* m_{i_{r-1}}, m_{i_r})^{-1}\prod_{l=2}^{r-1} C(m_{i_1}*\cdots *m_{i_{l-1}}, m_{i_l})^{-1}\prod_{l=r+1}^{j} C(m_{i_1}* \cdots *m_{i_{l-1}}, m_{i_l})^{-1}\right)\\
    &\phantom{==} \cdot
     C(m_{i_r}, m_{i_{r+1}}* \cdots *m_{i_j})C(m_{i_{r+1}}*\cdots *m_{i_j}, m_{i_r})^{-1}m_{i_r}\\
    &= \sum_{r=1}^{j}(-1)^{r-1}e_{\sigma_{i_r}}C(m_{i_1}*\cdots *m_{i_{r-1}}, m_{i_r})^{-1}\prod_{l=2}^{r-1} C(m_{i_1}*\cdots *m_{i_{l-1}}, m_{i_l})^{-1} \\
    &\phantom{==}\cdot \prod_{l=r+1}^{j}\left( C(m_{i_1}*\cdots *m_{i_{r-1}}*m_{i_{r+1}}* \cdots *m_{i_{l-1}}, m_{i_l})^{-1}C(m_{i_r}, m_{i_l})^{-1}\right)\\
    &\phantom{==}\cdot C(m_{i_r}, m_{i_{r+1}}* \cdots *m_{i_j})C(m_{i_{r+1}}*\cdots *m_{i_j}, m_{i_r})^{-1}m_{i_r}\\
    &= \sum_{r=1}^{j}(-1)^{r-1}e_{\sigma_{i_r}}C(m_{i_1}*\cdots *m_{i_{r-1}}, m_{i_r})^{-1}\prod_{l=2}^{r-1} C(m_{i_1}*\cdots * m_{i_{l-1}}, m_{i_l})^{-1} \\
    &\phantom{==}\cdot \prod_{l=r+1}^{j} C(m_{i_1} * \cdots * m_{i_{r-1}}*m_{i_{r+1}}* \cdots * m_{i_{l-1}}, m_{i_l})^{-1}\cancel{C(m_{i_r}, m_{i_{r+1}} * \cdots * m_{i_j})^{-1}}\\
    &\phantom{==}\cdot\cancel{C(m_{i_r}, m_{i_{r+1}} * \cdots * m_{i_j})}C(m_{i_{r+1}} * \cdots * m_{i_j}, m_{i_r})^{-1}m_{i_r}\\
\end{align*}

By comparing the coefficients we can see that $\psi \partial^{\bbT}(e_\sigma) = \partial^{\bbK}(\psi(e_\sigma)).$
\end{proof}

\begin{construction}
Let $I=(u_1,\ldots,u_m)$ be an ideal with linear quotients (for the given order on the generators) with regular decomposition function $g$ and complement factor $\kappa$. We define a complex $\bbHT_\b(u_1,\ldots,u_m)$ as follows: $\bbHT_i$ is the free module on the symbols $e(\sigma;u)$ with $u\in G(I)$ and $\sigma\subseteq\set(u)$ with $|\sigma|=i-1$ for $i\geq1$, while $\bbHT_0=R$. If $\sigma\neq\emptyset$, then the differential on $\bbHT_\b$ is given by

\begin{align*}
\partial^{\bbHT}(e(\sigma;u))=&-\sum_{t\in\sigma}e(\sigma_t;u)(-1)^{\alpha(\sigma;t)}C(x_{\sigma_t}*u,x_t)^{-1}x_t\\
&+
\sum_{t\in\sigma}e(\sigma_t;g(x_t*u))(-1)^{\alpha(\sigma;t)}C\left(x_{\sigma_t}*g(x_t*u),\kappa(x_t*u)\right)^{-1}\kappa(x_t*u),
\end{align*}
where we set $e(\sigma;u)=0$ if $\sigma\not\subseteq\set(u)$, and where $x_\sigma$ denotes the product of the variables $x_i$ with $i\in\sigma$. If $\sigma=\emptyset$, then
\[
\partial^\bbHT(e(\emptyset;u))=u.
\]
\end{construction}

The fact that $\bbHT_\b(u_1,\ldots, u_m)$ is a complex will follow from the proof of \Cref{thm:main}.

\begin{theorem}\label{thm:main}
Let $I=(u_1,\ldots,u_m)$ be an ideal with linear quotients (for the given order on the generators) and regular decomposition function. Then, the complex $\bbHT_\b(u_1,\ldots,u_m)$ is a minimal free resolution of $R/I$.
\end{theorem}

\begin{proof}
We prove by induction on $j$ that the complex $\bbHT_\b(u_1,\ldots,u_j)$ resolves the ideal $(u_1,\ldots,u_j)$, the case $j=1$ being clear. We denote by $\bbHT^{(j)}_\b$ the complex $\bbHT_\b(u_1,\ldots,u_j)$.

Let $I=(u_1,\ldots,u_m)$ have linear quotients with respect to the given order. Let $I_j=(u_1,\ldots, u_j)$ for $j=1,\ldots,m$ and $L_j=(I_j:u_{j+1})$ for $j=1,\ldots,m-1$. Let $\set(u_{j+1})=\{k_1,\ldots,k_l\}$, therefore $L_j=(x_{k_1},\ldots,x_{k_l})$. Consider the following short exact sequence
\[
0\rightarrow R/L_j\rightarrow R/I_j\rightarrow R/I_{j+1}\rightarrow0,
\]
where we are using that $I_{j+1}/I_j\cong R/L_j$.
The module $R/L_j$ is resolved by $\tbbT_\b(x_{k_1},\ldots,x_{k_l};u_{j+1})$. The module $R/I_j$ is resolved by $\bbHT^{(j)}_\b$. Let 
\[
\psi^{(j)}_\b:\tbbT_\b(x_{k_1},\ldots,x_{k_l};u_{j+1})\rightarrow \bbHT^{(j)}_\b
\]
be a lifting of the map $R/L_j\rightarrow R/I_j$. It suffices to show that $\bbHT^{(j+1)}_\b$ coincides with the mapping cone of $\psi^{(j)}_\b$, which we denote by $\cone(\psi_\b^{(j)})$.

Since $\bbHT_\b^{(j)}$ is a subcomplex of $\cone(\psi_\b^{(j)})$, we only need to check that
\[
\partial^{\cone(\psi_\b^{(j)})}(e(\sigma;u_{j+1}))=\partial^{\bbHT^{(j+1)}}(e(\sigma;u_{j+1})),\quad\text{for all }\sigma\subseteq\set(u_{j+1}).
\]
By definition of the differential of the mapping cone, it follows that
\[
\partial^{\cone(\psi_\b^{(j)})}(e(\sigma;u_{j+1}))=-\partial^{\tbbT}(e(\sigma;u_{j+1}))+\psi^{(j)}(e(\sigma;u_{j+1})),
\]
therefore, proving that the following choice of $\psi^{(j)}_\b$ is indeed a lifting will complete the argument. We define $\psi^{(j)}(e(\sigma;u_{j+1}))$ to be

\[
\sum_{t\in\sigma}e(\sigma_t;g(x_t*u_{j+1}))(-1)^{\alpha(\sigma;t)}C(x_{\sigma_t},\kappa(x_t*u_{j+1}))^{-1}C(g(x_t*u_{j+1}),\kappa(x_t*u_{j+1}))^{-1}\kappa(x_t*u_{j+1}),
\]
if $\sigma\neq\emptyset$ and $\psi^{(j)}(e(\emptyset;u_{j+1}))=u_{j+1}$ otherwise.
To verify this we must prove that $\psi^{(j)}\partial^{\tbbT}=\partial^{\bbHT^{(j)}}\psi^{(j)}$. 
In the following, if $s\neq t$, then $\sigma_{t,s}$ denotes the subset $\sigma$ with $s$ and $t$ removed. We start by computing $ \psi^{(j)}\partial^{\tbbT}(e(\sigma;u_{j+1}))=$
\begin{align*}
 &\hphantom{=} \sum_{t\in \sigma} \psi^{(j)}(e(\sigma_t;u_{j+1}))(-1)^{\alpha(\sigma;t)}C(x_{\sigma_t}*u_{j+1}, x_t)^{-1}x_t\\
    &= \sum_{t\in \sigma} \biggl( \sum_{s\in \sigma_t}e(\sigma_{t,s};g(x_s*u_{j+1}))(-1)^{\alpha(\sigma_t;s)}C(x_{\sigma_{t,s}},\kappa(x_s*u_{j+1}))^{-1} C(g(x_s*u_{j+1}),\kappa(x_s*u_{j+1}))^{-1}\\ 
    &\hspace{2cm} \cdot\kappa(x_s*u_{j+1}) \biggr) (-1)^{\alpha(\sigma;t)}C(x_{\sigma_t}*u_{j+1}, x_t)^{-1}x_t\\
    &= \sum_{t\in \sigma}\sum_{\substack{s\in \sigma_t\\ s<t}} e(\sigma_{t,s};g(x_s*u_{j+1}))(-1)^{\alpha(\sigma;t)+\alpha(\sigma;s)}C(x_{\sigma_{t,s}},\kappa(x_s*u_{j+1}))^{-1} C(g(x_s*u_{j+1}),\kappa(x_s*u_{j+1}))^{-1}\\* &\hspace{2cm}\cdot\kappa(x_s*u_{j+1})C(x_{\sigma_t}*u_{j+1}, x_t)^{-1}x_t\\
    &- \sum_{t\in \sigma}\sum_{\substack{s\in \sigma_t\\ s>t}} e(\sigma_{t,s};g(x_s*u_{j+1}))(-1)^{\alpha(\sigma;t)+\alpha(\sigma;s)}C(x_{\sigma_{t,s}},\kappa(x_s*u_{j+1}))^{-1} C(g(x_s*u_{j+1}),\kappa(x_s*u_{j+1}))^{-1}\\ &\hspace{2cm}\cdot\kappa(x_s*u_{j+1})C(x_{\sigma_t}*u_{j+1}, x_t)^{-1}x_t\\
    &= \sum_{t\in \sigma}\sum_{\substack{s\in \sigma_t\\ s<t}} e(\sigma_{t,s};g(x_s*u_{j+1}))(-1)^{\alpha(\sigma;t)+\alpha(\sigma;s)}C(x_{\sigma_{t,s}},\kappa(x_s*u_{j+1}))^{-1} C(g(x_s*u_{j+1}),\kappa(x_s*u_{j+1}))^{-1}\\ &\hspace{2cm}\cdot\kappa(x_s*u_{j+1})C(x_{\sigma_t}*u_{j+1}, x_t)^{-1}x_t\\
    &-\sum_{s\in \sigma}\sum_{\substack{t\in \sigma_s\\ s<t}} e(\sigma_{t,s};g(x_t*u_{j+1}))(-1)^{\alpha(\sigma;t)+\alpha(\sigma;s)}C(x_{\sigma_{s,t}},\kappa(x_t*u_{j+1}))^{-1} C(g(x_t*u_{j+1}),\kappa(x_t*u_{j+1}))^{-1}\\ &\hspace{2cm}\cdot\kappa(x_t*u_{j+1})C(x_{\sigma_s}*u_{j+1}, x_s)^{-1}x_s\\
    &= \sum_{t\in \sigma}\sum_{\substack{s\in \sigma_t\\ s<t}} e(\sigma_{t,s};g(x_s*u_{j+1}))(-1)^{\alpha(\sigma;t)+\alpha(\sigma;s)}C(x_{\sigma_{t,s}},\kappa(x_s*u_{j+1}))^{-1} C(g(x_s*u_{j+1}),\kappa(x_s*u_{j+1}))^{-1}\\ &\hspace{2cm} \cdot C(x_{\sigma_t}*u_{j+1}, x_t)^{-1}C(\kappa(x_s*u_{j+1}),x_t)\kappa(x_s*u_{j+1})*x_t\\
    &-\sum_{s\in \sigma}\sum_{\substack{t\in \sigma_s\\ s<t}}
    e(\sigma_{t,s};g(x_t*u_{j+1}))(-1)^{\alpha(\sigma;t)+\alpha(\sigma;s)}C(x_{\sigma_{s,t}},\kappa(x_t*u_{j+1}))^{-1} C(g(x_t*u_{j+1}),\kappa(x_t*u_{j+1}))^{-1}\\ &\hspace{2cm}\cdot C(x_{\sigma_s}*u_{j+1}, x_s)^{-1}C(\kappa(x_t*u_{j+1}),x_s)\kappa(x_t*u_{j+1})*x_s.
\end{align*}
Now we compute $\partial^{\bbHT^{(j)}}\psi^{(j)}(e(\sigma;u_{j+1}))$

\begin{align*}
 \partial^{\bbHT^{(j)}}\psi^{(j)}(e(\sigma;u_{j+1})) = \sum_{t \in \sigma} & \partial^{\bbHT^{(j)}}(e(\sigma_t;g(x_t*u_{j+1})))(-1)^{\alpha(\sigma;t)}C(x_{\sigma_t},\kappa(x_t*u_{j+1}))^{-1}  \\
    &\cdot C(g(x_t*u_{j+1}), \kappa(x_t*u_{j+1}))^{-1} \kappa(x_t*u_{j+1}),
\end{align*}

we expand $\partial^{\bbHT^{(j)}}(e(\sigma_t;g(x_t*u_{j+1})))$ below

\begin{align*}
-\sum_{s \in \sigma_t} &e(\sigma_{t,s};g(x_t*u_{j+1}))(-1)^{\alpha(\sigma_t;s)}C(x_{\sigma_{t,s}}*g(x_t*u_{j+1}),x_s)^{-1}x_s\\
      + \sum_{s \in \sigma_t} &e(\sigma_{t,s};g(x_s*g(x_t*u_{j+1})))
    (-1)^{\alpha(\sigma_t;s)}
    C \left( x_{\sigma_{t,s}}, \frac{x_s*g(x_t*u_{j+1})}{g(x_s*g(x_t*u_{j+1}))}\right)^{-1}\\
     & \cdot C\left( g(x_s*g(x_t*u_{j+1})), \frac{x_s*g(x_t*u_{j+1})}{g(x_s*g(x_t*u_{j+1}))} \right)^{-1}
    \frac{x_s*g(x_t*u_{j+1})}{g(x_s*g(x_t*u_{j+1}))}.
\end{align*}

Before continuing the proof we notice that it may happen that $\sigma_t\not\subseteq\set(g(x_t*u_{j+1}))$, in which case $e(\sigma_t;g(x_t*u_{j+1}))=0$ by convention. Therefore, the display above should also be zero.
Indeed, let $s\in\sigma_t$. If $\sigma_{t,s}\not\subseteq\set(g(x_t*u_{j+1}))$, then by regularity $\sigma_{t,s}\not\subseteq\set(g(x_s*g(x_t*u)))$, therefore the corresponding summands in the display above are zero. If $\sigma_{t,s}\subseteq\set(g(x_t*u_{j+1}))$, then $s\not\in\set(g(x_t*u_{j+1}))$, and therefore by \Cref{lem:lem1.8} $g(x_s*g(x_t*u_{j+1}))=g(x_t*u_{j+1})$. To show that the corresponding summands in the display above cancel, it suffices to show that $C(x_{\sigma_{t,s}}*g(x_t*u_{j+1}),x_s)^{-1}$ is equal to
\[
C \left( x_{\sigma_{t,s}}, \frac{x_s*g(x_t*u_{j+1})}{g(x_s*g(x_t*u_{j+1}))}\right)^{-1} \cdot C\left( g(x_s*g(x_t*u_{j+1})), \frac{x_s*g(x_t*u_{j+1})}{g(x_s*g(x_t*u_{j+1}))} \right)^{-1}.
\]
Indeed, since $g(x_s*g(x_t*u_{j+1}))=g(x_t*u_{j+1})$, this reduces to 
\[
C(x_{\sigma_{t,s}}*g(x_t*u_{j+1}),x_s)^{-1}=C(x_{\sigma_{t,s}}, x_s)^{-1} \cdot C( g(x_t*u_{j+1}), x_s)^{-1},
\]
which is true since $C$ is a bicharacter.

The element $\partial^{\bbHT^{(j)}}\psi^{(j)}(e(\sigma;u_{j+1}))$ is equal to

\begin{equation*}
\begin{split}
  \sum_{t \in \sigma} \Bigg( &-\sum_{\substack{s\in \sigma_t\\ s<t}}
e(\sigma_{t,s};g(x_t*u_{j+1}))(-1)^{\alpha(\sigma_t;s)}C(x_{\sigma_{t,s}}*g(x_t*u_{j+1}),x_s)^{-1}x_s\\
&-\sum_{\substack{s\in \sigma_t\\ s>t}} e(\sigma_{t,s};g(x_t*u_{j+1}))(-1)^{\alpha(\sigma_t;s)}C(x_{\sigma_{t,s}}*g(x_t*u_{j+1}),x_s)^{-1}x_s\\
&+ \sum_{\substack{s\in \sigma_t\\ s<t}} e(\sigma_{t,s};g(x_s*g(x_t*u_{j+1})))
(-1)^{\alpha(\sigma_t;s)}
C \left( x_{\sigma_{t,s}}, \frac{x_s*g(x_t*u_{j+1})}{g(x_s*g(x_t*u_{j+1}))}\right)^{-1} \\
 &\qquad \quad \cdot C\left( g(x_s*g(x_t*u_{j+1})), \frac{x_s*g(x_t*u_{j+1})}{g(x_s*g(x_t*u_{j+1}))} \right)^{-1}
\frac{x_s*g(x_t*u_{j+1})}{g(x_s*g(x_t*u_{j+1}))}\\ 
&+ \sum_{\substack{s\in \sigma_t\\ s>t}} e(\sigma_{t,s};g(x_s*g(x_t*u_{j+1})))
(-1)^{\alpha(\sigma_t;s)}
C \left( x_{\sigma_{t,s}}, \frac{x_s*g(x_t*u_{j+1})}{g(x_s*g(x_t*u_{j+1}))}\right)^{-1} \\
 &\qquad \quad \cdot C\left( g(x_s*g(x_t*u_{j+1})), \frac{x_s*g(x_t*u_{j+1})}{g(x_s*g(x_t*u_{j+1}))} \right)^{-1}
\frac{x_s*g(x_t*u_{j+1})}{g(x_s*g(x_t*u_{j+1}))}\Bigg)\\
&\quad \cdot(-1)^{\alpha(\sigma;t)}C(x_{\sigma_t},\kappa(x_t*u_{j+1}))^{-1} C(g(x_t*u_{j+1}), \kappa(x_t*u_{j+1}))^{-1} \kappa(x_t*u_{j+1}).
\end{split}
\end{equation*}
Distributing the external sum, rewriting $\alpha(\sigma_t;s)$ and reordering the variables of the monomials yields

\begin{equation*}
\begin{split}
-\sum_{t \in \sigma}\sum_{\substack{s\in \sigma_t\\ s<t}}
&e(\sigma_{t,s};g(x_t*u_{j+1}))(-1)^{\alpha(\sigma;t)+\alpha(\sigma;s)}C(x_{\sigma_{t,s}}*g(x_t*u_{j+1}),x_s)^{-1}\\
&\cdot C(x_{\sigma_t},\kappa(x_t*u_{j+1}))^{-1} C(g(x_t*u_{j+1}), \kappa(x_t*u_{j+1}))^{-1} C(x_s, \kappa(x_t*u_{j+1})) x_s*\kappa(x_t*u_{j+1})\\
+\sum_{s \in \sigma}\sum_{\substack{t\in \sigma_s\\ s<t}} &e(\sigma_{s,t};g(x_s*u_{j+1}))(-1)^{\alpha(\sigma;t)+\alpha(\sigma;s)}C(x_{\sigma_{s,t}}*g(x_s*u_{j+1}),x_t)^{-1}\\
&\cdot C(x_{\sigma_s},\kappa(x_s*u_{j+1}))^{-1} C(g(x_s*u_{j+1}), \kappa(x_s*u_{j+1}))^{-1} C(x_t, \kappa(x_s*u_{j+1})) x_t*\kappa(x_s*u_{j+1})\\
+ \sum_{t \in \sigma}\sum_{\substack{s\in \sigma_t\\ s<t}} &e(\sigma_{t,s};g(x_s*g(x_t*u_{j+1}))) (-1)^{\alpha(\sigma;t)+\alpha(\sigma;s)}
C \left( x_{\sigma_{t,s}}, \frac{x_s*g(x_t*u_{j+1})}{g(x_s*g(x_t*u_{j+1}))}\right)^{-1} \\
 &\cdot C\left( g(x_s*g(x_t*u_{j+1})), \frac{x_s*g(x_t*u_{j+1})}{g(x_s*g(x_t*u_{j+1}))} \right)^{-1}\\
&\cdot C(x_{\sigma_t},\kappa(x_t*u_{j+1}))^{-1} 
C(g(x_t*u_{j+1}), \kappa(x_t*u_{j+1}))^{-1}\\ 
&\cdot C\left(\frac{x_s*g(x_t*u_{j+1})}{g(x_s*g(x_t*u_{j+1}))}, \kappa(x_t*u_{j+1})\right)
\frac{x_s*g(x_t*u_{j+1})}{g(x_s*g(x_t*u_{j+1}))}*\kappa(x_t*u_{j+1})\\
- \sum_{s \in \sigma}\sum_{\substack{t\in \sigma_s\\ s<t}} &e(\sigma_{s,t};g(x_t*g(x_s*u_{j+1})))(-1)^{\alpha(\sigma;t)+\alpha(\sigma;s)}
C \left( x_{\sigma_{s,t}}, \frac{x_t*g(x_s*u_{j+1})}{g(x_t*g(x_s*u_{j+1}))}\right)^{-1} \\
&\cdot C\left( g(x_t*g(x_s*u_{j+1})), \frac{x_t*g(x_s*u_{j+1})}{g(x_t*g(x_s*u_{j+1}))} \right)^{-1}\\
&\cdot C(x_{\sigma_s},\kappa(x_s*u_{j+1}))^{-1} 
C(g(x_s*u_{j+1}), \kappa(x_s*u_{j+1}))^{-1}\\ 
&\cdot C\left(\frac{x_t*g(x_s*u_{j+1})}{g(x_t*g(x_s*u_{j+1}))}, \kappa(x_s*u_{j+1})\right)
\frac{x_t*g(x_s*u_{j+1})}{g(x_t*g(x_s*u_{j+1}))}*\kappa(x_s*u_{j+1}).
\end{split}
\end{equation*}

To finish proving that $\psi^{(j)}\partial^{\tbbT}(e(\sigma;u_{j+1}))=\partial^{\bbHT^{(j)}}\psi^{(j)}(e(\sigma;u_{j+1}))$, we will show that the first double sum in the display above coincides with the second double sum in the expansion of $\psi^{(j)}\partial^{\tbbT}(e(\sigma;u_{j+1}))$, the second double sum in the display above coincides with the first double sum in the expansion of $\psi^{(j)}\partial^{\tbbT}(e(\sigma;u_{j+1}))$, and then we will show that the third and fourth double sums in the display above cancel.

 We first simplify the product of the $C$'s in the first double summation above

\begin{align*}
&C(x_{\sigma_{t,s}}*g(x_t*u_{j+1}),x_s)^{-1}C(x_{\sigma_t},\kappa(x_t*u_{j+1}))^{-1} C(g(x_t*u_{j+1}), \kappa(x_t*u_{j+1}))^{-1} C(x_s, \kappa(x_t*u_{j+1}))\\
=&C(x_{\sigma_{t,s}}*g(x_t*u_{j+1}),x_s)^{-1}C\left(\frac{x_s}{x_{\sigma_t}*g(x_t*u_{j+1})}, \kappa(x_t*u_{j+1})\right)\\
=&C(x_{\sigma_{t,s}}*g(x_t*u_{j+1}),x_s)^{-1}C(x_{\sigma_{t,s}}*g(x_t*u_{j+1}), \kappa(x_t*u_{j+1}))^{-1}\\
=&C\left( x_{\sigma_{t,s}} * g(x_t*u_{j+1}),x_s * \frac{x_t * u_{j+1}}{g(x_t*u_{j+1})} \right)^{-1}.
\end{align*}

We simplify the product of the $C$'s in the second double summation in the expansion of $\psi^{(j)}\partial^{\tbbT}(e(\sigma;u_{j+1}))$

\begin{align*}
&C(x_{\sigma_{s,t}},\kappa(x_t*u_{j+1}))^{-1} C(g(x_t*u_{j+1}),\kappa(x_t*u_{j+1}))^{-1} 
    C(x_{\sigma_s}*u_{j+1}, x_s)^{-1}C(\kappa(x_t*u_{j+1}),x_s)\\
    =&C(x_{\sigma_{s,t}}*g(x_t*u_{j+1}),\kappa(x_t*u_{j+1}))^{-1}C(x_{\sigma_s}*u_{j+1}, x_s)^{-1}C(\kappa(x_t*u_{j+1}),x_s)\\
    =&C(x_{\sigma_{s,t}}*g(x_t*u_{j+1}),\kappa(x_t*u_{j+1}))^{-1}C \left( \frac{1}{x_{\sigma_s}*u_{j+1}}, x_s \right)
    C\left( \frac{x_t * u_{j+1}}{g(x_t*u_{j+1})},x_s \right)\\
    =&C(x_{\sigma_{s,t}}*g(x_t*u_{j+1}),\kappa(x_t*u_{j+1}))^{-1}C(x_{\sigma_{s,t}}*g(x_t*u_{j+1}), x_s)^{-1}\\
    =&C\left( x_{\sigma_{t,s}} * g(x_t*u_{j+1}),x_s * \frac{x_t * u_{j+1}}{g(x_t*u_{j+1})} \right)^{-1}.
\end{align*}

This is enough to show that the first double sum in the expansion of $\partial^{\bbHT^{(j)}}\psi^{(j)}(e(\sigma;u_{j+1}))$ coincides with the second double sum in the expansion of $\psi^{(j)}\partial^{\tbbT}(e(\sigma;u_{j+1}))$. A similar argument shows that the second double sum in the expansion of $\partial^{\bbHT^{(j)}}\psi^{(j)}(e(\sigma;u_{j+1}))$ coincides with the first double sum in the expansion of $\psi^{(j)}\partial^{\tbbT}(e(\sigma;u_{j+1}))$.

Now we show that the third and fourth double sums in the expansion of $\partial^{\bbHT^{(j)}}\psi^{(j)}(e(\sigma;u_{j+1}))$ cancel. Since
\[
\frac{x_t*g(x_s*u_{j+1})}{g(x_t*g(x_s*u_{j+1}))}*\kappa(x_s*u_{j+1})=\frac{x_s*x_t*u_{j+1}}{g(x_t*g(x_s*u_{j+1}))},
\]
which is symmetric in $s$ and $t$, it suffices to show that the product of the $C$'s is the same in both double summations. We first simplify the product of the $C$'s in the third double summation
\begin{align*}
&C \left( x_{\sigma_{t,s}}, \frac{x_s*g(x_t*u_{j+1})}{g(x_s*g(x_t*u_{j+1}))}\right)^{-1} 
  C\left( g(x_s*g(x_t*u_{j+1})), \frac{x_s*g(x_t*u_{j+1})}{g(x_s*g(x_t*u_{j+1}))} \right)^{-1}\\
  &\cdot
C(x_{\sigma_t},\kappa(x_t*u_{j+1}))^{-1} 
C(g(x_t*u_{j+1}), \kappa(x_t*u_{j+1}))^{-1}
 C\left(\frac{x_s*g(x_t*u_{j+1})}{g(x_s*g(x_t*u_{j+1}))}, \kappa(x_t*u_{j+1})\right)\\
=&C\left(x_{\sigma_{t,s}}* g(x_s*g(x_t*u_{j+1})), \frac{x_s*g(x_t*u_{j+1})}{g(x_s*g(x_t*u_{j+1}))} \right)^{-1}\\
&\cdot C\left(\frac{1}{x_{\sigma_t}*g(x_t*u_{j+1})} *\frac{x_s*g(x_t*u_{j+1})}{g(x_s*g(x_t*u_{j+1}))}, \kappa(x_t*u_{j+1})\right)\\
=&C\left(x_{\sigma_{t,s}}* g(x_s*g(x_t*u_{j+1})), \frac{x_s*g(x_t*u_{j+1})}{g(x_s*g(x_t*u_{j+1}))} \right)^{-1}C\left(x_{\sigma_{t,s}}* g(x_s*g(x_t*u_{j+1})), \kappa(x_t*u_{j+1})\right)^{-1}\\
=&C\left(x_{\sigma_{t,s}}* g(x_s*g(x_t*u_{j+1})), \frac{x_s*x_t*u_{j+1}}{g(x_s*g(x_t*u_{j+1}))}\right)^{-1}.
\end{align*}

To conclude the proof it suffices to show that the product of the $C$'s in the fourth double summation simplifies to the same scalar. Indeed,

\begin{align*}
&C \left( x_{\sigma_{s,t}}, \frac{x_t*g(x_s*u_{j+1})}{g(x_t*g(x_s*u_{j+1}))}\right)^{-1}
C\left( g(x_t*g(x_s*u_{j+1})), \frac{x_t*g(x_s*u_{j+1})}{g(x_t*g(x_s*u_{j+1}))} \right)^{-1}\\
&\cdot C(x_{\sigma_s},\kappa(x_s*u_{j+1}))^{-1} 
C(g(x_s*u_{j+1}), \kappa(x_s*u_{j+1}))^{-1}
 C\left(\frac{x_t*g(x_s*u_{j+1})}{g(x_t*g(x_s*u_{j+1}))}, \kappa(x_s*u_{j+1})\right)\\
=&C\left( x_{\sigma_{s,t}}*g(x_t*g(x_s*u_{j+1})), \frac{x_t*g(x_s*u_{j+1})}{g(x_t*g(x_s*u_{j+1}))} \right)^{-1} \\
&\cdot C\left(\frac{1}{x_{\sigma_s}*g(x_s*u_{j+1})}*\frac{x_t*g(x_s*u_{j+1})}{g(x_t*g(x_s*u_{j+1}))}, \kappa(x_s*u_{j+1})\right)\\
=&C\left( x_{\sigma_{s,t}}*g(x_t*g(x_s*u_{j+1})), \frac{x_t*g(x_s*u_{j+1})}{g(x_t*g(x_s*u_{j+1}))} \right)^{-1}C\left(x_{\sigma_{s,t}}*g(x_t*g(x_s*u_{j+1})), \kappa(x_s*u_{j+1})\right)^{-1}\\
=&C\left(x_{\sigma_{s,t}}* g(x_t*g(x_s*u_{j+1})), \frac{x_t*x_s*u_{j+1}}{g(x_t*g(x_s*u_{j+1}))}\right)^{-1},
\end{align*}
which is the same as the scalar in the third double summation by \Cref{lem:swap}.
\end{proof}

\begin{definition}
If $I$ is an ideal with linear quotients and regular decomposition function, then the \emph{skew Herzog-Takayama resolution} of $R/I$ is $\bbHT_\b$.
\end{definition}

\begin{remark}
It follows from (a skew version of) \cite[Lemma 1.1]{EK}, and the Remark after it, that $I$ is stable if and only if $\max g(m)\leq\min\kappa(m)$ for all $m\in M(I)$, where $\max g(m)$ is the largest index of a variable dividing $g(m)$ and $\min\kappa(m)$ the smallest index of a variable dividing $\kappa(m)$. Therefore, if $I$ is stable, then $C(g(m),\kappa(m))=1$ for all $m\in M(I)$, and the resolution $\bbHT_\b$ coincides with the resolution constructed in \cite[Theorem 3.4]{EKskew}. 
\end{remark}

\begin{definition}
A squarefree monomial ideal $I$ in $R$ is \emph{squarefree stable} if for every minimal generator $u$ of $I$ and all $i<\max u$ with $i\not\in\supp u$, one has that $\frac{x_i*u}{x_{\max u}}\in I$.
\end{definition}

\begin{remark}
In the commutative case squarefree stable ideals were studied by Aramova, Herzog and Hibi in \cite{sqrfreeStable}, who constructed a minimal free resolution. If one orders the minimal generators of a squarefree stable ideal $I$ with respect to the reverse degree lexicographic order, then one can see that $I$ has linear quotients with $\set u=\{i\mid i<\max u,i\not\in\supp u\}$ for every minimal generator $u$. It can be verified that the decomposition function is regular, therefore the Herzog-Takayama resolution can be used to resolve these ideals, and in fact it reduces to the Aramova-Herzog-Hibi resolution. Similarly, the skew Herzog-Takayama resolution can be used to minimally resolve squarefree stable ideals in skew polynomial rings.
\end{remark}

\begin{definition}
A squarefree ideal $I$ of a skew polynomial ring $R$ generated in a single degree is called \emph{matroidal} if for every pair of minimal generators $\mathbf{x}^\mathbf{a},\mathbf{x}^\mathbf{b}$, one has that if $a_i>b_i$ for some $i$, then there exists a $j$ with $a_j<b_j$ such that $\frac{x_j*\mathbf{x}^\mathbf{a}}{x_i}$ is another minimal generator.
\end{definition}

\begin{remark}
It was proved in \cite[Lemma 1.3 and Theorem 1.10]{cone} that matroidal ideals in commutative polynomial rings have linear quotients with respect to the reverse lexicographic order; moreover, their decomposition function is regular. The same argument works in the skew case, showing that the skew Herzog-Takayama resolution resolves matroidal ideals in skew polynomial rings.
\end{remark}

\begin{remark}
Explicit resolutions of matroidal ideals in commutative polynomial rings have been constructed, in different terms, by Reiner and Welker \cite{ReinWelk} and Novik, Postnikov and Sturmfels \cite{NovPosSturm}.
\end{remark}

In the next corollaries we will not be assuming that the ideal $I$ has a regular decomposition function, but we can still apply \Cref{thm:main} since the hypothesis on the decomposition function is only used to find an explicit formula for the differential. Without that hypothesis the resolution can still be constructed via an iterated mapping cone, but the differential is more mysterious. The subsequent corollaries are only concerned with Betti numbers, so we do not need to know the differential of the resolution. Let $I$ be an ideal with linear quotients, when working with the bigraded Poincar\'{e} series of $R/I$ as an $R$-module, which we will denote by $P_{R/I}^R(s,t)$ where $s$ is the variable for the homological degree and $t$ the variable for the $\mathbb{Z}$-degree, we ignore the $G$-grading.

\begin{corollary}
If $I$ is an ideal with linear quotients in a skew polynomial ring $R$, then the bigraded Poincar\'{e} series of $R/I$ is given by
\[
P_{R/I}^R(s,t)=1+\sum_{u\in G(I)}(1+s)^{|\set u|}st^{\deg u}.
\]
\end{corollary}

In the next corollary $\beta_i(I)$ will denote the $i$th Betti number of $I$
\begin{corollary}
Let $I$ be an ideal with linear quotients in a skew polynomial ring, then
\[
\beta_i(I)=\sum_{u\in G(I)}\binom{|\set u|}{i}.
\]
\end{corollary}

In the next corollary $\mathrm{pd}_R$ denotes the projective dimension of an $R$-module.

\begin{corollary}
If $I$ is an ideal with linear quotients in a skew polynomial ring, then
\[
\mathrm{pd}_R\;R/I=1+\max_{u\in G(I)}|\set u|.
\]
\end{corollary}

\begin{chunk}
Let $M$ be a $\mathbb{Z}$-graded $R$-module. The \emph{Tor-regularity} of $M$ is
\[
\mathrm{Tor.reg}_R\;M=\mathrm{sup}\{j-i\mid \beta_{i,j}(M)\neq0\},
\]
where $\beta_{i,j}(M)$ is the rank of the free module in the minimal free resolution of $M$ that has homological degree $i$ and $\mathbb{Z}$-degree $j$.

One can also define the \emph{Castelnuovo-Mumford regularity} of $M$, as in \cite[Definition 2.1]{Jo4}. We denote this invariant by $\mathrm{CM.reg}_R\;M$.

\end{chunk}

As for the commutative case, the following corollary follows directly from the construction of the skew Herzog-Takayama resolution of an ideal with linear quotients.

\begin{corollary}\label{cor:TorReg}
Let $I$ be an ideal with linear quotients in a skew polynomial ring $R$. Then the Tor-regularity of $I$ is given by
\[
\mathrm{Tor.reg}_R\;I=\max\{\mathrm{deg}(u)\mid u\in G(I)\}.
\]
\end{corollary}

\begin{corollary}
Let $I$ be an ideal with linear quotients in a skew polynomial ring $R$. Then the Castelnuovo-Mumford regularity of $I$ is given by
\[
\mathrm{CM.reg}_R\;I=\max\{\mathrm{deg}(u)\mid u\in G(I)\}.
\]
\end{corollary}

\begin{proof}
It follows from \cite[Theorem 5.4]{Dong} and \cite[Corollary 4.14]{Yek} that over a skew polynomial ring $R$, the Castelnuovo-Mumford regularity of $I$ and the Tor-regularity of $I$ coincide. Now, one concludes by invoking \Cref{cor:TorReg}.
\end{proof}

We conclude the paper with an example.

\begin{example}
Consider the ideal $I =(x_1x_2,x_1x_3,x_2x_3,x_2x_4)$. One can check that this ideal has linear quotients for the given order on the generators. Moreover, a simple calculation shows that
\[
\set(x_1x_2)=\emptyset,\quad \set(x_1x_3)=\{2\},\quad\set(x_2x_3)=\{1\},\quad \text{and}\quad\set(x_2x_4)=\{1,3\}.
\]
One can also easily check that the decomposition function is regular.
The skew Herzog-Takayama resolution of the ideal $I$ has the following structure
\[
\hspace{-0.25in}\begin{tikzpicture}[baseline=(current  bounding  box.center)]
 \matrix (m) [matrix of math nodes,row sep=3em,column sep=2.5em,minimum width=2em] {
0&Re(1,3;x_2x_4)&\begin{matrix}Re(2;x_1x_3)\oplus Re(1;x_2x_3)\\\oplus\\Re(1;x_2x_4)\oplus Re(3;x_2x_4)\end{matrix}&\begin{matrix}Re(\emptyset;x_1x_2)\oplus Re(\emptyset;x_1x_3)\\\oplus\\Re(\emptyset;x_2x_3)\oplus Re(\emptyset;x_2x_4)\end{matrix}&R\\};
\path[->] (m-1-1) edge (m-1-2);
\path[->] (m-1-2) edge  node[above]{$\partial_3$} (m-1-3);
\path[->] (m-1-3) edge  node[above]{$\partial_2$} (m-1-4);
\path[->] (m-1-4) edge  node[above]{$\partial_1$} (m-1-5);
\end{tikzpicture}
\]
where $\partial_1$ is the obvious map, while the differentials $\partial_2$ and  $\partial_3$ are defined as follows

\[   \partial_2 =
        \begin{pmatrix}
        x_3 & x_3 & x_4 & 0\\
        -q_{2,3}x_2 & 0 & 0 & 0\\
        0 & -q_{1,2}q_{1,3}x_1 & 0 & x_4\\
        0 & 0 & -q_{1,2}q_{1,4}x_1 & -q_{3,4}x_3
        \end{pmatrix},  
     \partial_3 = 
        \begin{pmatrix}
        0\\ 
        -x_4\\ 
        q_{3,4}x_3\\ 
        -q_{1,2}q_{1,3}q_{1,4}x_1
        \end{pmatrix}.
\]
We show the calculation of $\partial_3(e(1,3;x_2x_4))$, the columns of $\partial_2$ are computed similarly. The differential $\partial_3(e(1,3;x_2x_4))$ is equal to
 
\begin{align*}
       &-e(3;x_2x_4) (-1)^{\alpha(\sigma, 1)} C(x_2x_3x_4, x_1)^{-1}x_1 - e(1;x_2x_4) (-1)^{\alpha(\sigma, 3)} C(x_1x_2x_4, x_3)^{-1}x_3 \\
      & +e(3;g(x_1x_2x_4)) (-1)^{\alpha(\sigma, 1)} C\left( x_3,\frac{x_1x_2x_4}{g(x_1x_2x_4)}\right)^{-1} C\left(g(x_1x_2x_4), \frac{x_1x_2x_4}{g(x_1x_2x_4)}\right)^{-1} \frac{x_1x_2x_4}{g(x_1x_2x_4)} \\
      & + e(1;g(x_2x_3x_4)) (-1)^{\alpha(\sigma, 3)} C\left(x_1,\frac{x_2x_3x_4}{g(x_2x_3x_4)}\right)^{-1} C\left(g(x_2x_3x_4), \frac{x_2x_3x_4}{g(x_2x_3x_4)}\right)^{-1} \frac{x_2x_3x_4}{g(x_2x_3x_4)}.
\end{align*}

Since $g(x_1x_2x_4) = x_1x_2$ and $g(x_2x_3x_4) = x_2x_3$, the previous display simplifies to

\begin{align*}
     \partial_3(e(1,3;x_2x_4)) = &- e(3;x_2x_4)\overbrace{C(x_2x_3x_4, x_1)^{-1}}^{q_{1,2}q_{1,3}q_{1,4}} x_1 + e(1;x_2x_4) \overbrace{C(x_1x_2x_4, x_3)^{-1}}^{q_{3,4}}x_3\\
      & + e(3;x_1x_2) C(x_3,x_4)^{-1}C(x_1x_2, x_4)^{-1}x_4 - e(1;x_2x_3) C(x_1,x_4)^{-1} C(x_2x_3, x_4)^{-1}x_4.
\end{align*}

Note that 
\[
C(x_3,x_4)=C(x_1x_2,x_4)=C(x_1,x_4)=C(x_2x_3,x_4)=1.
\]

Therefore
\[
\partial_3(e(1,3;x_2x_4))=-q_{1,2}q_{1,3}q_{1,4}e(3;x_2x_4)x_1+q_{3,4}e(1;x_2x_4)x_3+e(3;x_1x_2)x_4-e(1;x_2x_3)x_4.
\]

Since $3\not\in\set(x_1x_2)$, it follows by our convention that $e(3;x_1x_2)=0$, finally yielding
\[
\partial_3(e(1,3;x_2x_4))=-q_{1,2}q_{1,3}q_{1,4}e(3;x_2x_4)x_1+q_{3,4}e(1;x_2x_4)x_3-e(1;x_2x_3)x_4.
\]

\end{example}

\bibliographystyle{amsplain}
\bibliography{biblio}
\end{document}